\documentclass[11pt, reqno]{amsart}

\usepackage{latexsym}
\usepackage{amssymb}
\usepackage{mathrsfs}
\usepackage{amsmath}
\usepackage{fancybox,color}
\usepackage{enumerate}
\usepackage[latin1]{inputenc}
\usepackage{eurosym}

\newcommand{\kom}[1]{}
 \def\1{\raisebox{2pt}{\rm{$\chi$}}}
\def\a{{\bf a}}

\newtheorem{theorem}{Theorem}[section]

\newtheorem{proposition}[theorem]{Proposition}

\newcommand{\R}{{\mathbb R}}

 \newcommand{\eps}{{\varepsilon}}
 \def\1{\raisebox{2pt}{\rm{$\chi$}}}
 
\newcommand{\abs}[1]{\left|#1\right|}
\newcommand{\norm}[1]{\left|\left|#1\right|\right|}
\newcommand{\Rn}{\mathbb{R}^n}

\def\vint_#1{\mathchoice%
          {\mathop{\kern 0.2em\vrule width 0.6em height 0.69678ex depth -0.58065ex
                  \kern -0.8em \intop}\nolimits_{\kern -0.4em#1}}%
          {\mathop{\kern 0.1em\vrule width 0.5em height 0.69678ex depth -0.60387ex
                  \kern -0.6em \intop}\nolimits_{#1}}%
          {\mathop{\kern 0.1em\vrule width 0.5em height 0.69678ex
              depth -0.60387ex
                  \kern -0.6em \intop}\nolimits_{#1}}%
          {\mathop{\kern 0.1em\vrule width 0.5em height 0.69678ex depth -0.60387ex
                  \kern -0.6em \intop}\nolimits_{#1}}}
\def\vintslides_#1{\mathchoice%
          {\mathop{\kern 0.1em\vrule width 0.5em height 0.697ex depth -0.581ex
                  \kern -0.6em \intop}\nolimits_{\kern -0.4em#1}}%
          {\mathop{\kern 0.1em\vrule width 0.3em height 0.697ex depth -0.604ex
                  \kern -0.4em \intop}\nolimits_{#1}}%
          {\mathop{\kern 0.1em\vrule width 0.3em height 0.697ex depth -0.604ex
                  \kern -0.4em \intop}\nolimits_{#1}}%
          {\mathop{\kern 0.1em\vrule width 0.3em height 0.697ex depth -0.604ex
                  \kern -0.4em \intop}\nolimits_{#1}}}

\newcommand{\kint}{\vint}
\newcommand{\intav}{\vint}
\newcommand{\aveint}[2]{\mathchoice%
          {\mathop{\kern 0.2em\vrule width 0.6em height 0.69678ex depth -0.58065ex
                  \kern -0.8em \intop}\nolimits_{\kern -0.45em#1}^{#2}}%
          {\mathop{\kern 0.1em\vrule width 0.5em height 0.69678ex depth -0.60387ex
                  \kern -0.6em \intop}\nolimits_{#1}^{#2}}%
          {\mathop{\kern 0.1em\vrule width 0.5em height 0.69678ex depth -0.60387ex
                  \kern -0.6em \intop}\nolimits_{#1}^{#2}}%
          {\mathop{\kern 0.1em\vrule width 0.5em height 0.69678ex depth -0.60387ex
                  \kern -0.6em \intop}\nolimits_{#1}^{#2}}}
\newcommand{\ud}{\, d}
\newcommand{\half}{{\frac{1}{2}}}

\newcommand{\ol}{\overline}

\newcommand{\dist}{\operatorname{dist}}

\newcommand{\vp}{\varphi}

\newcommand{\tr}{\operatorname{tr}}
\renewcommand{\a}{\alpha}

\newcommand{\re}{\mathbb{R}}
\newcommand{\rn}{\mathbb{R}^n}
\newcommand{\ut}{u_{\eps}}
\begin{document}

\title[Gradient walk and $p$-harmonic functions]{Gradient walk and  $p$-harmonic functions}

\author[Luiro]{Hannes Luiro}
\address{Department of Mathematics and Statistics, University of
Jyv\"askyl\"a, PO~Box~35 (MaD), FI-40014 Jyv\"askyl\"a, Finland}
\email{hannes.s.luiro@jyu.fi}

\author[Parviainen]{Mikko Parviainen}
\address{Department of Mathematics and Statistics, University of
Jyv\"askyl\"a, PO~Box~35 (MaD), FI-40014 Jyv\"askyl\"a, Finland}
\email{mikko.j.parviainen@jyu.fi}


\date{October, 2015}

\begin{abstract}
We consider a class of stochastic processes and establish its connection to $p$-harmonic functions. In particular, we obtain stochastic approximations that converge uniformly to a $p$-harmonic function, with an explicit convergence rate, and also obtain a precise diffusion representation in continuous time.   The main difficulty is how to deal  with the zero set of the gradient of the underlying function.   
\end{abstract}

\maketitle


\section{Introduction}

A connection between a stochastic game called a tug-of-war with noise and $p$-Laplace equation
\[
\begin{split}
\operatorname{div}(\abs{\nabla u}^{p-2}\nabla u)=0,
\end{split}
\]
 was first discovered by Peres and Sheffield in \cite{peress08}, see also \cite{peresssw09}. In one of the key results they show that if the gradient of a $p$-harmonic payoff function is nonvanishing and a player follows the gradient strategy, then the game value is close  to the $p$-harmonic function.  

However, the zeros of the gradient pose deep problems both in the theory of PDEs and  tug-of-war games. This is the main difficulty we encounter in this paper. 
To be more precise, we study a gradient walk i.e.\ we fix gradient strategies for both the players, but we do not assume that the gradient is nonvanishing. A motivation for such an approach stems partly from desire to understand a local
 behavior of stochastic processes related to $p$-harmonic functions as well as  to obtain a diffusion representation in continuous time for  $p$-harmonic functions. 

The main results show that the expectation under the gradient walk approximates a $p$-harmonic function, Theorems \ref{thm:main-1} (uniform convergence with respect to the step size) and \ref{thm:main-1-rate} (explicit rate with respect to the step size), without any assumptions on the zero set of the gradient of the underlying $p$-harmonic function. Naturally, the set where the gradient is small requires a special attention in terms of how we define the process. 

If the zero set of the gradient is known to be a finite set of points, we develop a technique which is more flexible with respect to how we define the process at the zero set of the gradient. A natural choice at those points is to choose the next point at random according to a uniform probability distribution. Such a process  has the approximation property as shown in Theorem \ref{thm:main-many}, and the continuous time version gives an exact diffusion presentation of a $p$-harmonic function, Theorem \ref{thm:main-1-cont-time}.  Such stochastic approximations seem to be deeply connected to the structure of the zero set of the gradient for a $p$-harmonic function.  In the plane, the zero set of the gradient is known to be discrete \cite{bojarskii87}, but in higher dimensions, understanding structures of the zero set is a difficult open problem.



\subsection{Background}
Taking the average over  the usual Taylor expansion
\[
u(x+y)=u(x)+\nabla u(x)\cdot y
+\frac12 D^2u(x)y\cdot y+O(|y|^3),
\]
over $B(0,\eps)$, we obtain
\begin{equation}
\label{exp.lapla} u(x) -  \intav_{B(0,\eps)} u(x+y)\ud y = -
  \frac{\eps^2}{2(n+2)}\Delta u (x)  + O(\eps^3),
\end{equation}
when $u$ is smooth. Above we used the notation
\[
\kint_{ B(0,\eps)} u(x+y) \ud y=\frac{1}{\abs{B(0,\eps)}}\int_{ B(0,\eps)} u(x+y) \ud y.
\]
On the other hand, by assuming $\nabla u(x)\neq 0$, evaluating the Taylor expansion with $y=\pm \eps\frac{\nabla u(x)}{\abs{\nabla
u(x)}}$, and summing up we get
\begin{equation}\label{exp.infty}
\begin{split}
u(x)&-\frac{1}{2} \left\{ u \left(x+\eps\frac{\nabla u(x)}{\abs{\nabla
u(x)}} \right)+ u\left(x-\eps\frac{\nabla u(x)}{\abs{\nabla u(x)}}\right)\right\}\\
&= - \frac{\eps^2}{2} \Delta^N_\infty u (x) + O(\eps^3),
\end{split}
\end{equation}
where $\Delta_\infty^N u=\abs{\nabla u}^{-2}\sum_{i,j=1}^{n} u_{ij} u_{i}u_{j}$ is the normalized infinity Laplace operator, and $u_i, u_{ij}$ denote the first and second derivatives respectively.
Next if we multiply \eqref{exp.lapla} and \eqref{exp.infty} by $\beta=(2+n)/(p+n),\ 2\le p<\infty$, and $1-\beta$, and add
up the formulas, we get the normalized $p$-Laplace operator on the right hand side i.e.\
\[
\begin{split}
u(x) =&\frac{1-\beta}{2} \left\{ u \left(x+\eps\frac{\nabla u(x)}{\abs{\nabla
u(x)}} \right)+ u\left(x-\eps\frac{\nabla u(x)}{\abs{\nabla u(x)}}\right)\right\}\\
&+ \beta
\kint_{B(0,\eps)} u(x+y) \ud y + C\eps^2 \Delta_p^N u (x)+O (\eps^3)
\end{split}
\]
as $\eps \to 0$, where $\Delta_p^N u=(\Delta u+(p-2)\Delta_\infty^N u )$ denotes the normalized $p$-Laplacian. The equation $\Delta_p^N u=0$ gives the same solutions as the usual $p$-Laplace operator, see \cite{juutinenlm01, kawohlmp12, julinj12}. If $\Delta_p^N u=0$, then the above formula yields
\begin{equation}
\label{eq:p-Taylor}
\begin{split}
u(x) =&\frac{1-\beta}{2} \left\{ u \left(x+\eps\frac{\nabla u(x)}{\abs{\nabla
u(x)}} \right)+ u\left(x-\eps\frac{\nabla u(x)}{\abs{\nabla u(x)}}\right)\right\}\\
&+ \beta
\kint_{B(0,\eps)} u(x+y) \ud y +O (\eps^3),
\end{split}
\end{equation}
as $\eps\to 0$. A variant of this formula can then be used to characterize the $p$-harmonic functions \cite{manfredipr10}.

The above formula suggests the following Markov chain, which we call the gradient walk: when at $x$ step to $x+\eps\frac{\nabla u(x)}{\abs{\nabla
u(x)}}$ or to $x-\eps\frac{\nabla u(x)}{\abs{\nabla u(x)}}$ with probability $(1-\beta)/2$, respectively, or to a point chosen according to a uniform probability distribution on $B(x,\eps)$ with probability $\beta$. Similarly, this defines one step probability measures at every point, which again induce a probability measure on the space of sequences according to the Kolmogorov construction. Take $u_\eps(x)$ to be the expectation with respect to this probability measure when starting at $x$, stopping when exiting a domain (stopping time $\tau$), and taking the boundary values from a smooth $p$-harmonic function $u$ with $\nabla u\neq 0$ i.e.
\begin{equation}
\label{eq:expect}
\begin{split}
\ut(x):=\mathbb E^x\big[u(x_{\tau})\big].
\end{split}
\end{equation}
 For more details of the stochastic background, see for example \cite{luirops14}. 

If we drop the error term in (\ref{eq:p-Taylor}) and replace $u$ by $u_\eps$, then the resulting formula holds by the Markov property, see \cite{meynt09} Section 3.4.2, and tells us, how to compute the expectation at the point $x$: this is done by summing up the three possible outcomes with the corresponding probabilities. This is the key formula needed in establishing a connection with the expected value and the corresponding $p$-harmonic function in the case  $\nabla u\neq 0$. 

The formula  (\ref{eq:p-Taylor}) also suggest a version of a dynamic programming principle and a version of a tug-of-war with noise with good symmetry properties, see for example in \cite{manfredipr12, luirops13, luirops14}.  

\section{Preliminaries}
\label{sec:setup}
We consider a domain $B(0,1)\subset \Rn,\ n\ge 2,$ and a $p$-harmonic function $u:B(0,1+\gamma)\to\re, \gamma>0$, throughout the paper. Further, let $x\in B(0,1)$, $\eps,\eta>0$, $\abs{\nabla u}>\eta$, $\beta=(2+n)/(p+n)$ and  set
\begin{equation}
\label{eq:1-mitta}
 \mu_{x,1}= \beta\mathcal{L}_{B(x,\eps)}+\frac{1-\beta}{2}\big(\delta_{x+\eps\frac{\nabla u(x)}{|\nabla u(x)|}}+
         \delta_{x-\eps\frac{\nabla u(x)}{|\nabla u(x)|}}\big)\,,
\end{equation}
where $\mathcal L_{B(x,\eps)}$ denotes the uniform distribution in $B(x,\eps)\subset\rn$ and $\delta_x$ the Dirac measure at $x$.
In this notation, the equation (\ref{eq:p-Taylor}) can be written as 
\begin{equation}\label{eq:smoothpoint}
\bigg{|}\int_{\rn} u(y)\,d\mu_{x,1}(y)\,-u(x)\,\bigg{|}\,\leq\,C\eps^3\,.
\end{equation}


Another technical tool we use repeatedly is the fact that the expected value of the distance to a \textit{fixed} point increases at every step in the gradient walk at a certain rate. 
\begin{proposition}\label{aux1}
For all $x\in\rn$, $\nu\in S^{n-1}=\partial B(0,1)$, $\eps>0$ and $\beta\in (0,1]$,  it holds that  
\begin{equation}
\label{eq:drift-eq}
\frac{1-\beta}{2}\big(|x+\eps \nu|+|x-\eps \nu|\big)+\beta\intav_{B(x,\eps)}|y|dy\,\geq\,|x|+C(n)\beta\frac{\eps^2}{2(|x|+\eps)}\,,
\end{equation}
where $C(n)\to 1$ as $n\to\infty$.
\end{proposition}


\begin{proof}
Observe first that $|x+\eps \nu|+|x-\eps \nu|\geq 2|x|$  by the triangle inequality. Assume for now  $x\not=0$, and denote by $\langle x\rangle^{\bot}$ the space orthogonal to $x$, by 
$y^{\bot}$ the projection of $y\in\rn$ onto $\langle x\rangle^{\bot}$ and by $P(y)$ the reflection of $y$ with respect to $\langle x\rangle^{\bot}$. 
Then $P$ is a linear isometry, thus we can write
\begin{align*}
\intav_{B(x,\eps)}|y|dy=&\intav_{B(0,\eps)}\frac{|x+y|+|x+P(y)|}{2}\,dy\,\geq \intav_{B(x,\eps)}\frac{|x+y+x+P(y)|}{2}dy\\
=&\intav_{B(0,\eps)}|x+y^{\bot}|\,dy\,,
\end{align*}
where we used the triangle inequality and the fact $y+P(y)=2y^{\bot}$. By Pythagoras' theorem, the following estimate holds for all $a,b\in\rn$, $a\cdot b=0\,$: 
\begin{align*}
|a+b|&=(|a|^2+|b|^2)^{\frac{1}{2}}=|a|+\int_{|a|^2}^{|a|^2+|b|^2}\frac{1}{2\sqrt{t}}dt\,\\
&\geq |a|+\frac{|b|^2}{2(|a|^2+|b|^2)^{\frac{1}{2}}}\,=\,|a|+\frac{|b|^2}{2|a+b|}\,.
\end{align*}
This estimate implies that
\begin{align*}
\intav_{B(0,\eps)}|x+y^{\bot}|\,dy\,&\geq |x|+\intav_{B(0,\eps)}\frac{|y^{\bot}|^2}{2|x+y^{\bot}|}\,dy\,\\
&\geq |x|+\frac{1}{2(|x|+\eps)}\intav_{B(0,\eps)}|y^{\bot}|^2\,dy\,.
\end{align*}
Finally, it is easy to see that for $0<c<1$ we have  
\begin{align*}
\intav_{B(0,\eps)}|y^{\bot}|^2\,dy\,&=\,\frac{n-1}{n}\intav_{B(0,\eps)}|y|^2\,dy\,\geq\,\frac{n-1}{n}\frac{|B(0,\eps)\setminus 
B(0,c\eps)|}{|B(0,\eps)|}|c\eps|^2\,\\
&=\frac{n-1}{n}(1-c^n)|c\eps|^2\,.\\
\end{align*}
By choosing $c$ suitably depending on $n$ yields the claim by combining the above estimates.

 If $x=0$, take nonzero approximating sequence $x_i\to 0$, and use the above result.
\end{proof}

\section{Immediate evaluation of error}

Above we considered the domain $B(0,1)$, and a  $p$-harmonic function $u:B(0,1+\gamma)\to\re$, $\gamma>0$, under the assumption of the non vanishing gradient. However, as already pointed out the difficulty lies in the zero set of the gradient, and now we drop the assumption $\nabla u\neq 0$. Define the one step probability measures
\begin{equation}\nonumber
 \mu_{x,2}=\delta_{x+\eps\frac{x}{|x|}}\text{ if }x\not= 0\,\text{ and }\mu_{0,2}=\delta_{\eps|e_1|}\,, 
\end{equation}
and
\begin{equation}\nonumber
\mu_{x}=\begin{cases}
         \mu_{x,1},\text{ if }|\nabla u(x)|\geq \eta\,\,,\,\text{ and }\\
         \mu_{x,2}, \text{ if }|\nabla u(x)|< \eta\,,
        \end{cases}
\end{equation}
where $\mu_{x,1}$ was defined in (\ref{eq:1-mitta}). The first version of a gradient walk which is a Markov chain $\{x_0,x_1,\ldots\}\subset \Rn$ starting at $x_0$ is determined by the transition probabilities
$\mu_{x}\,$.
Moreover, for $x\in B(0,1)$, we define a value of the gradient walk to be given by
\begin{equation}
\label{eq:value-walk-1}
\begin{split}
\ut(x):=u_{\eps,\eta,u}(x):=\mathbb{E}^x[u(x_{\tau})]
\end{split}
\end{equation}
where $\tau$ is the first exit time from $B(0,1)$, and the notation is also otherwise the same as in (\ref{eq:expect}), and $u_{\eps}:=u$ outside $B(0,1)$. 
Since there is a fixed non-zero probability that the walk approaches $\partial B(0,1)$ at least by 
a step comparable to $\eps$, it follows that $\tau$ is finite almost surely, and the same holds for other stopping times defined in this paper as well.

The name of the section reflects the fact that in the bad set $\{ |\nabla u(x)|< \eta \}$ we obtain comparison result implying Theorem \ref{thm:main-1} in one step, cf.\ (\ref{eq:impossible2}) below. In contrast, in Section  \ref{sec:delayed} we define the bad set slightly differently and wait until we exit the bad set. The immediate evaluation of error uses properties of the stochastic process in the bad set in a subtle way, but on the other hand  does not require any assumptions on the structure of the zero set of the gradient. It also allows us to obtain explicit convergence rate. 

Our first main theorem states that the value of the gradient walk converges uniformly to the underlying $p$-harmonic function. 
\begin{theorem}
\label{thm:main-1}
Let  $u:B(0,1+\gamma)\to\re$ be a $p$-harmonic function, $2\le p<\infty$. Let $\eta>0$ and let $u_\eps$  be the value of the gradient walk given by (\ref{eq:value-walk-1}). Then for any $C>1$, there is $\eps_0>0$ such that\begin{equation}\nonumber
\norm{u_{\eps}-u}_{L^{\infty}(B(0,1))}\leq C \eta 
\end{equation}
for all $0<\eps<\eps_0$.
\end{theorem}
\begin{proof}
Let us denote $g(x):=|\ut(x)-u(x)\,|$. Choose an auxiliary comparison function $f(x)=c\eta(c-|x|)$, where $c>1$. We will establish that $g\leq f$, which implies the claim. 
To this end, assume the opposite, so that
\begin{equation}\label{eq:counter}
M:=\sup_{x\in B(0,1)}(g(x)-f(x))>0\,.
\end{equation}
Suppose that $M$ above is achieved at $x_0\in B(0,1)$ up to an arbitrary small error term $\kappa>0$
\begin{equation}\label{eq:maxpoint}
g(x_0)-f(x_0)\geq M-\kappa\,.
\end{equation} 
First we consider the case $|\nabla u(x_0)|\geq \eta$.
By the definition of $\ut$ and (\ref{eq:smoothpoint}) it follows that 
\begin{align}
\label{eq:key-smooth-comp}
g(x_0)-&f(x_0)\nonumber\\
&=|\ut(x_0)-u(x_0)|-f(x_0)=\big|\int_{\rn}\ut(y)\,d\mu_{x_0,1}-u(x_0)\,\big|-f(x_0)\nonumber \\
&\leq\int_{\rn}|\ut(y)-u(y)|\,d\mu_{x_0,1}(y)\,+C\eps^3\,-f(x_0)\,\nonumber\\
& =\int_{\rn}g(y)-f(y)\,d\mu_{x_0,1}(y)\,+\int_{\rn}f(y)\,d\mu_{x_0,1}(y)\,-f(x_0)+C\eps^3\,\nonumber\\
& \leq M+\int_{\rn}f(y)\,d\mu_{x_0,1}(y)\,-f(x_0)+C\eps^3\,.
\end{align}
Combining this with (\ref{eq:maxpoint}) implies
\begin{equation}\label{eq:impossible1}
f(x_0)\leq \int_{\rn}f(y)\,d\mu_{x_0,1}(y)\,+C\eps^3+\kappa\,.
\end{equation}
However, Proposition \ref{aux1} implies that this can not be true for any $x\in B(0,1)$ 
for our cone-function $f$ if $\eps$ and $\kappa$ are small enough. 

Consider then the case  $|\nabla u(x_0)|< \eta$ (and $x_0\not=0$, the case $x_0=0$ is similar). In this case it follows that
\begin{align*}
g(x_0)&-f(x_0)=|\ut(x_0)-u(x_0)|-f(x_0)\\
&=\big|\int_{\rn}\ut(y)\,d\mu_{x_0,2}-u(x_0)\,\big|-f(x_0)\\
&=|\ut(x_0+\eps\frac{x_0}{|x_0|})-u(x_0)|-f(x_0)\,\\
&\leq|\ut(x_0+\eps\frac{x_0}{|x_0|})-u(x_0+\eps\frac{x_0}{|x_0|})|+|u(x_0+\eps\frac{x_0}{|x_0|})-u(x_0)|-f(x_0)\\
& = g(x_0+\eps\frac{x_0}{|x_0|})-f(x_0+\eps\frac{x_0}{|x_0|})+f(x_0+\eps\frac{x_0}{|x_0|})-f(x_0)\\
&\,\,\,\,\,\,+|u(x_0+\eps\frac{x_0}{|x_0|})-u(x_0)|\,\\
& \leq M-c\eta\eps+|u(x_0+\eps\frac{x_0}{|x_0|})-u(x_0)|\,, 
\end{align*}
implying  by (\ref{eq:maxpoint}) that
\begin{equation}\label{eq:impossible2}
|u(x_0+\eps\frac{x_0}{|x_0|})-u(x_0)|\geq c\eta\eps -\kappa\,.
\end{equation}
\sloppy However by using the fact that $c>1$, $|\nabla u(x_0)|< \eta$, and $u$ is $C^1$, it follows that (\ref{eq:impossible2}) can not be true
for $\eps\le \eps_0$. Summing up, we have shown that (\ref{eq:counter}) cannot be valid, and thus $\norm{u_{\eps}-u}_{L^{\infty}(B(0,1))}\leq c^2\eta $. 
\end{proof}

We could also have used stochastic approach in the proof above, consider the expectation over a single step, estimate the accumulation of the error, and to use optional stopping theorem, cf.\ \cite[Theorem 2.4]{peress08} or \cite[Theorem 4.1]{manfredipr12}. Our technique provides an efficient alternative in the present setting. The key incredients are approximation property (\ref{eq:smoothpoint}) and $C^1$ regularity of the limit $u$ as well as the drift property (\ref{eq:drift-eq}). The fact that $u$ is $p$-harmonic is only utilized through these properties.

\subsection{Convergence rate}

In this section,  we aim at obtaining an explicit convergence rate of the gradient walk with respect to $\eps$.  First recall the notation with multi-index $\sigma$ and $\a\in (0,1)$
\[
\begin{split}
\norm{u}_{C^{1,\a}(B)}&=\sum_{\abs{\sigma}\le 1}||D^{\sigma} u||_{C^{\a}(B)},\qquad 
\norm{u}_{C^{\a}(B)}=\sup_{B}\abs{u}+\abs{u}_{C^{\a}(B)},\\
\abs{u}_{C^{\a}(B)}&=\sup_{x,y\in B,x\neq y}\frac{\abs{u(x)-u(y)}}{\abs{x-y}}^{\a}.
\end{split}
\]
We will also use $\norm{u}_{C^{3,\a}}$ below and the definition is analogous to the one above. 

Let $u:B(0,1+\gamma)\to \R$, $\gamma>0$ be a $p$-harmonic function. Then there is $\a=\a(n,p)\in (0,1)$ and $C=C(n,p,\gamma,\norm{u}_{L^{\infty}(B(0,1+\gamma))})$ such that
\begin{equation}
\label{eq:c1a}
\begin{split}
\norm{u}_{C^{1,\a}(B(0,1))}\le C,
\end{split}
\end{equation}
see for example \cite{uraltseva68,uhlenbeck77,evans82,dibenedetto83,tolksdorf84}. Note that $p$-Laplacian degenerates when the gradient vanishes, and $C^{1,\a}$-regularity is optimal. 
We will also use another standard estimate. 
\begin{theorem}
\label{thm:C3}
Let $u$ be $p$-harmonic function in $B(0,1)$ such that
\[
\begin{split}
\half \le \abs{\nabla u(x)} \text{ in }B(0,1).
\end{split}
\] 
Then there is $\a=\a(n,p)\in (0,1)$ and $C=C(n,p,\norm{u}_{L^{\infty}(B(0,1))})$ such that
$u\in C^{3,\a}(B(0,\half))$ and
\[
\begin{split}
\norm{u}_{C^{3,\a}(B(0,\half))}\le C.
\end{split}
\] 
\end{theorem}

The theorem is well-known. It is based for example on using Schauder estimates (see for example \cite{gilbargt01}) combined with the estimate (\ref{eq:c1a}).
To be more precise, consider the normalized $p$-Laplacian $a_{ij}(\nabla u)u_{ij}=0$ using the Einstein summation convention where $u_{ij}$ denotes the second derivatives and
\[
\begin{split}
a_{ij}(q)=\delta_{ij}+(p-2)\frac{q_iq_j}{\abs{q}^2},
\end{split}
\]
and $\delta_{ij}=1$ if $i=j$ and zero otherwise. Observe that $a_{ij}(\cdot)$ is smooth when $q$ us bounded away from zero. Then we have $a_{ij}(\nabla u)\in C^{\a}$ by  (\ref{eq:c1a}), and $u\in C^{2,\a}$ by the Schauder theory and $\half \le \abs{\nabla u}$.
Then the heuristic idea  is differentiating the equation and denoting by $w=D_{\nu} u$ the directional derivative to the direction $\nu$, we get
\[
\begin{split}
a_{ij,k}(\nabla u)w_k u_{ij}+a_{ij}(\nabla u)w_{ij}=0,
\end{split}
\]
where $a_{ij,k}$ denotes  the derivative of $a_{ij}$ with respect to the $k$th variable.  The $C^{2,\a}$-estimate for $w$ depends on the $C^{\a}$-norm for the coefficients, that is, on the $C^{2,\a}$-estimate for $u$. 
Thus the estimate in Theorem \ref{thm:C3} holds with a uniform coefficient $C=C(n,p,\norm{u}_{L^{\infty}(B(0,1))})$. An alternative approach can be built on the divergence form equation and the weak formulation.

Now we consider the gradient walk defined by the one step probability measure
\begin{equation}
 \mu_{x,2}=\delta_{x+\eps\frac{x}{|x|}}\text{ if }x\not= 0\,\text{ and }\mu_{0,2}=\delta_{\eps|e_1|}\,, 
\end{equation}
and
\begin{equation}
\label{eq:explicit-cut-measure}
\mu_{x}=\begin{cases}
         \mu_{x,1},\text{ if }|\nabla u(x)|\geq \eps^{\a'},\\
         \mu_{x,2},\text{ if }|\nabla u(x)|< \eps^{\a'},
        \end{cases}
\end{equation}
where $\a'<\a/2$ is fixed,  $\a$ is given by (\ref{eq:c1a}), and $\mu_{x,1}$ by (\ref{eq:1-mitta}). As before, we set 
\[
\begin{split}
u_{\eps}(x):=\mathbb{E}^x[u(x_{\tau})]\,.
\end{split}
\]
\begin{theorem}
\label{thm:main-1-rate}
Let  $u:B(0,1+\gamma)\to\re$ be a $p$-harmonic function, $2\le p<\infty$,  and let $u_\eps$  be the value of the gradient walk given above. Then  for any $C>1$ there is $\eps_0>0$ 
such that
\begin{equation}
\norm{u_{\eps}-u}_{L^{\infty}(B(0,1))}\leq C \eps^{\a'} 
\end{equation}
for all $0<\eps<\eps_0$.
\end{theorem}
\begin{proof}
\sloppy
Let $\lambda>0$ be small enough such that $r: =(\lambda/(2C'))^{1/\a}<\gamma$, where $C'$ is the one in the estimate (\ref{eq:c1a}). Suppose that $x\in B(0,1)$ such that $\abs{\nabla u(x)}=\lambda$. It  follows from (\ref{eq:c1a}) that $B(x,r)\subset \{z\in B(0,1+\gamma)\,:\,\frac{\lambda}{2} \le \abs{\nabla u(z)}\le 2\lambda\}$. Then set  
$$
v(y)=\frac{u(yr+x)-u(x)}{\lambda r}
$$ 
so that  $\half \le \abs{\nabla v(y)}\le 2$ in $B(0,1)$, and $v(0)=0$. It follows that $\norm{v}_{L^{\infty}(B(0,1))}\le 2$.
We are in the position of using the estimate from Theorem \ref{thm:C3} for $v$ with a constant $C(n,p)$. 
In particular,  let  $y\in B(0,\frac14)$, $\tilde \eps=\eps/r$ and 
\[
\begin{split}
v\left(y+\frac{\nabla v(y)}{\abs{\nabla v(y)}}\tilde \eps \right)=\frac{u\left(y'+\frac{\nabla u(y')}{\abs{\nabla u(y')}} \frac{\eps}{r} r \right)-u(x)}{\lambda r},
\end{split}
\]
where $y'=x+yr$. Then we have
\[
\begin{split}
&C\tilde \eps^3\\
&\ge \abs{ \beta \kint_{B(0,\tilde \eps)} v\ud y+\frac{1-\beta}2\Bigg\{v\left(y+\frac{\nabla  v(y)}{\abs{\nabla v(y)}} \tilde \eps \right)+v\left(y-\frac{\nabla v(y)}{\abs{\nabla v(y)}} \tilde \eps \right)\Bigg\}-v(y)}\\
&=\bigg|\frac{\beta}{\lambda r} \kint_{B(x,\eps)} u\ud y'+\frac{1-\beta}{2\lambda r}\Bigg\{u\left(y'+\frac{\nabla u(y')}{\abs{\nabla u(y')}} \eps \right)+u\left(y'-\frac{\nabla u(y')}{\abs{\nabla u(y')}}\eps \right)\Bigg\}\\
&\hspace{28 em}-u(y')\bigg|,
\end{split}
\]
and the estimate holds for $y'\in B(x,\frac{r}4)$. Multiplying  the both sides by $\lambda r$, we see that the desired formula holds with the error $C(\eps/r)^3\lambda r=C\eps^3\lambda r^{-2}$.
If we fix $\lambda=\eps^{\a'}$, the error term reads as 
\[
\begin{split}
 \lambda^{1-\frac{2}{\a}} \eps^3=\eps^{\a'(1-\frac{2}{\a})} \eps^3
\end{split}
\]
up to a constant.
On the other hand, if we have $\lambda>\eps^{\a'}$, then the above argument also works and  with  a smaller error term.  To sum up, a counterpart of  formula (\ref{eq:smoothpoint}) holds in the form
\begin{equation}\label{eq:smoothpoint2}
\bigg{|}\int_{\rn} u(y)\,d\mu_{x,1}(y)\,-u(x)\,\bigg{|}\,\leq\,C\eps^{2+\delta}\,,
\end{equation}
where $\delta =1+\a'(1-\frac{2}{\a})>0$. 

Next let $c>1$ and set $f(x)=c\eps^{\a'}(c-\abs{x}).$
Further, denote $g(x):=|\ut(x)-u(x)\,|$, let $\kappa>0$ and assume again thriving for a contradiction that there is $x_0\in B(0,1)$ such that 
\[
\begin{split}
g(x_0)-f(x_0)+\kappa\ge \sup_{x\in B(0,1)}(g(x)-f(x))>0.
\end{split}
\]
Suppose first that $|\nabla u(x_0)|\geq \eps^{\a'}$. By calculation (\ref{eq:key-smooth-comp})  using (\ref{eq:smoothpoint2}), we see that 
$$
f(x_0)\leq \int_{\rn}f(y)\,d\mu_{x_0,1}(y)\,+C\eps^{2+\delta }+\kappa.
$$ However, by Proposition \ref{aux1}  
$$
f(x_0)-C\eps^{2+\a' }\ge \int_{\rn}f(y)\,d\mu_{x_0,1}(y).
$$ The contradiction follows if $\a'<\a/2$.

Then assume that  $|\nabla u(x_0)|< \eta=\eps^{\a'}$.
By (\ref{eq:c1a}), it also holds that 
\[
\begin{split}
\sup_{y\in B(x,\eps)}\abs{u(y)-u(x)-\nabla u(x)\cdot (y-x)}\le C'r^{1+\a}
\end{split}
\]
implying 
\begin{equation}
\label{eq:c1beta}
\begin{split}
\sup_{y\in B(x,\eps)}\abs{u(y)-u(x)}&\le C'\eps^{1+\a}+\eps \abs{\nabla u(x)}\\
&< C'\eps^{1+\a}+\eps^{1+\a'}\le (C'\eps^{\a-\a'}+1) \eps^{1+\a'}.
\end{split}
\end{equation}
Next recall $f$ and observe that there is $\eps_0>0$ 
such that the right hand side of (\ref{eq:c1beta}) is smaller than $c\eps^{1+\a'}$ for every $0<\eps<\eps_0$. Using this, we conclude similarly as in the proof of Theorem \ref{thm:main-1}.
\end{proof}

\section{Delayed evaluation of error}

\label{sec:delayed}

In this section, we study a slightly different technique of showing that the gradient walk approximates a $p$-harmonic function with gradient vanishing in a finite set of points. In contrast to the previous sections, we take a ball neighborhood of a zero set of the gradient and then in that set we delay the evaluation of the error until we exit from the set. The good point in this method is that there is a lot of freedom to choose the probability measure as long as the resulting stochastic process exits the ball neighborhood almost surely. In a sense, a very natural choice is to use the random walk in the zeros of the gradient and otherwise to utilize  the gradient directions. This is the choice we use below. A counterpart of this choice in continuous time allows us to establish a diffusion representation of $p$-harmonic functions. 

\subsection{Diffusion representation of $p$-harmonic functions}

Next we study a continuous time diffusion process related to a $p$-harmonic function. This may be combared to the Feynman-Kac formula in the classical setting. A continuous time tug-of-war game has been previously studied by Atar and Budhiraja in \cite{atarb10} and in the context of option pricing in \cite{nystrompb}. 

Let $u:B(0,1+\gamma)\to \R$, $\gamma>0,$ be a $p$-harmonic function such that  $\abs{\nabla u(x)}>0$ whenever $x\neq x_1$, where $x_1 \in B(0,1)$.
Consider $\mathcal A(x)$, an $n\times n$ matrix, whose entries are given by
\[
\begin{split}
a_{ij}(y)=
\half \begin{cases}
\delta_{ij}+(p-2)\abs{\nabla u(y)}^{-2}u_{i}(y)u_{j}(y), & \text{ when }y\neq x_1\\
\delta_{ij}, &\text{ when }y=x_1,
\end{cases}
\end{split}
\]
where $u_i$ denotes the $i$th partial derivative, and $\delta_{ij}=1$ if $i=j$ and zero otherwise. In $B(0,1+\gamma)\setminus \{x_1\}$, the $p$-harmonic  function $u$ is real analytic and it holds that $\sum_{i,j=1}^na_{ij} u_{ij} =0$.
Also observe that $\min(p-1,1)\abs{\xi}^2\le  \sum_{i,j=1}^n a_{ij}\xi_i\xi_j\le \max(p-1,1)\abs{\xi}^2$ so that $\mathcal A$ is uniformly elliptic.
Thus 
there exists a  strong Markov family $(X(t),\mathbb P_x)$ of solutions to the martingale problem, see  \cite{bass98} Chapter 6, Stroock and  Varadhan \cite{stroockv06}  Chapter 6 and 12
as well as \cite{krylov73}.
Define  a stopping time $\tau$ as a first exit time from $B(0,1)$. Then 
\begin{equation}
\label{eq:ito}
\begin{split}
\vp&(X (t \wedge\tau))
-\vp(x) -\int_0^{t \wedge\tau}      \sum_{i,j=1}^na_{ij}(X(s)) \vp_{ij}(X(s)) \ud s,
\end{split}
\end{equation}
where $t\wedge \tau=\min(t,\tau)$, is a $\mathbb P_x$-martingale for all $\phi \in C^{2}(B(0,1+\delta))$. 
Moreover, $\tau$ as well as the stopping times defined in the proof below are finite a.s., see for example  Exercise 40.1 in \cite{bass11}.  

 We may define
\[
\begin{split}
v(x)=\mathbb E^x[u(X(\tau))].
\end{split}
\]

\begin{theorem}
\label{thm:main-1-cont-time}
Let $u:B(0,1+\gamma)\to \R$ be a $p$-harmonic function, $1<p<\infty$, such that  $\abs{\nabla u(x)}>0$ whenever $x\neq x_1$, where $x_1 \in B(0,1)$. Further, let $v$ be as above. Then $v=u$ in $B(0,1)$.\end{theorem}
\begin{proof}
Let $\eta>0$ and set  $f(x):=\eta(1-\abs{x-x_1})$ and $z:=x-x_1$.  It holds whenever $x\neq x_1$ that 
\[
\begin{split}
\nabla f(x)=-\eta\frac{z}{\abs z},\qquad D^2f(x)=\frac{\eta}{\abs z} \left( \frac{z\otimes z}{\abs{z}^2}-I \right),
\end{split}
\]
and
\begin{equation}
\label{eq:f-strict}
\begin{split}
&\sum_{i,j=1}^n a_{ij}(x)f_{ij}(x)\\
&=
\tr\Bigg(\left(I+(p-2)\frac{\nabla u(x)\otimes \nabla u(x)}{\abs{\nabla u(x)}^2}\right )\frac{\eta}{\abs z} \left( \frac{z\otimes z}{\abs{z}^2}-I \right)\Bigg)\\
&=\frac{\eta}{\abs z} \tr\Bigg( \frac{z\otimes z}{\abs{z}^2}+(p-2)\frac{\nabla u(x)}{\abs{\nabla u(x)}}(\frac{\nabla u(x)}{\abs{\nabla u(x)}})^T\frac{z}{\abs z} (\frac{z}{\abs z})^T \\
&\hspace{12 em}-I-(p-2)\frac{\nabla u(x)\otimes \nabla u(x)}{\abs{\nabla u(x)}^2}\Bigg)\\
&=\frac{\eta}{\abs z}\Big(1+(p-2)(\frac{\nabla u(x)}{\abs{\nabla u(x)}}\cdot \frac{z}{\abs z})^2-n-(p-2)\Big)\le \frac{\eta}{\abs z}(1-n)<0.
\end{split}
\end{equation}
Denote $g(x):=\abs{v(x)-u(x)}$. We aim at showing $g(x)-f(x)\le 0$. Thriving for a contradiction, let $\kappa\in (0,\eta/2)$ and suppose  that there is $M>0$ and $x_0\in B(0,1)$ such that
\begin{equation}
\label{eq:counter-cont}
\begin{split}
g(x_0)-f(x_0)+\kappa\ge M:=\sup_{B(0,1)}( g(x)-f(x)).
\end{split}
\end{equation}

Then we fix $r$ and again split the argument into two cases: Now if $x_0\in B(0,1)\setminus \ol B(x_1,r)$, take a ball $B(x_0,\delta)\subset B(0,1)\setminus \ol B(x_1,r)$, and stopping time $\tau^*$ as the first exit time from $B(x_0,\delta)$. Then by the strong Markov property, we have
\begin{equation}
\label{eq:v-expt}
\begin{split}
v(x_0)=\mathbb E^{x_0}[v(X(\tau^*))].
\end{split}
\end{equation}
Moreover,  $u$ is a smooth $p$-harmonic function at the vicinity of $B(x_0,\delta)$ so that
\[
\begin{split}
u(X(t\wedge \tau^*))&-u(x_0)-\int_0^{t\wedge \tau^*} \sum_{i,j=1}^n a_{ij}(X(s))u_{ij}(X(s)) \ud s\\
&=u(X(t\wedge \tau^*))-u(x_0)
\end{split}
\]  
is a martingale. By the optional stopping theorem
\begin{equation}
\begin{split}
\label{eq:u-expt}
u(x_0)=\mathbb E^{x_0}[u(X(\tau^*))].
\end{split}
\end{equation}
Similarly,
\[
\begin{split}
f(X(t\wedge \tau^*))-f(x_0)-\int_0^{t\wedge \tau^*} \sum_{i,j=1}^n a_{ij}(X(s))f_{ij}(X(s)) \ud s&
\end{split}
\]  
is a martingale, and the optional stopping theorem combined with (\ref{eq:f-strict}) gives
\begin{equation}
\label{eq:f-expt}
\begin{split}
-\kappa +f(x_0)>\mathbb E^{x_0}[f(X(\tau^*))],
\end	{split}
\end{equation}
for all small enough $\kappa$.
Thus by (\ref{eq:counter-cont}), (\ref{eq:v-expt}), (\ref{eq:u-expt}) and (\ref{eq:f-expt})
\[
\begin{split}
M-\kappa&\le g(x_0)-f(x_0)=\abs{v(x_0)-u(x_0)}-f(x_0)\\
&\le\mathbb E^{x_0}[\abs{v(X(\tau^*))-u(X(\tau^*))}-f(X(\tau^*))]+\mathbb E^{x_0}[f(X(\tau^*))]-f(x_0)\\
&=\mathbb E^{x_0}[g(X(\tau^*))-f(X(\tau^*))]+\mathbb E^{x_0}[f(X(\tau^*))]-f(x_0)<M-\kappa,
\end{split}
\] 
a contradiction.

If $x_0\in \ol B(x_1,r)$, we define a stopping time $\ol \tau$ as the first exit time from $B(x_1,2r)$. First we observe that 
\[
\begin{split}
\mathbb E^{x_0}[\abs{u(X(\ol \tau)-u(x_0)}]&\le \mathbb E^{x_0}[\abs{u(X(\ol \tau)-u(x_1)}+\abs{u(x_1)-u(x_0)}]\le 2C r^{1+\a}.
\end{split}
\]
Then we estimate 
\[
\begin{split}
M-\kappa&\le g(x_0)-f(x_0)=\abs{v(x_0)-u(x_0)}-f(x_0)\\
&=\abs{\mathbb E^{x_0}[v(X(\ol \tau))]-u(x_0)}-f(x_0)\\
&\le \mathbb E^{x_0}[\abs{v(X(\ol \tau))-u(X(\ol \tau))}]-f(x_0)+Cr^{1+\a}\\
&\le \mathbb E^{x_0}[\abs{v(X(\ol \tau))-u(X(\ol \tau))}-f(X(\ol \tau))]+\mathbb E^{x_0}[f(X(\ol \tau))]-f(x_0)+Cr^{1+\a}\\
&\le M+\mathbb E^{x_0}[f(X(\ol \tau))]-f(x_0)+Cr^{1+\a},
\end{split}
\]
where the second step holds by the Markov property for $v$, and third step by $C^{1,\a}$ regularity of $u$ and $\nabla u(x_1)=0$. 
This implies that
\[
\begin{split}
-\kappa+f(x_0)\le \mathbb E^{x_0}[f(X(\ol \tau))]+ Cr^{1+\a}
\end{split}
\]
which is a contradiction by the form of $f$ when $\kappa$ and $r$ are small enough. Since $\eta>0$ was arbitrary this proves the claim.
\end{proof}

\subsection{Gradient vanishing in a finite  set of points}
Next we return to the discrete time setting. If the gradient vanishes in a finite  set of points $E$, then the above method can easily be  modified to prove that the gradient walk approximates the original 
$p$-harmonic function by only changing the process on $E$. 
More precisely, using the notation from Section \ref{sec:setup}, let
us define for a given $p$-harmonic function $u:B(0,1+\gamma)\to\re$, $\gamma>0$, $x\in B(0,1)$, and $\eps>0$ that   
\begin{equation}\nonumber
\mu_{x}=\begin{cases}
         \mu_{x,1},&\text{ if }|\nabla u(x)|> 0\,\,,\,\text{ and }\\
         \mathcal{L}_{B(x,\eps)},&\text{ if }|\nabla u(x)|= 0\,.
        \end{cases}
\end{equation}
Then set similarly as before
\begin{equation}\nonumber
u_{\eps}(x):=\mathbb{E}^x[u(x_{\tau})]\,.
\end{equation}

\begin{theorem}
\label{thm:main-many}
Let $u:B(0,1+\gamma)\to \R$ be a $p$-harmonic function such that   $\abs{\nabla u(x)}>0$ outside a finite set of points $E\subset B(0,1)$. Then 
\begin{equation}\nonumber
|u_{\eps}-u|\to 0 \text{ uniformly in $B(0,1)$ as }\eps\to 0\,. 
\end{equation}
\end{theorem}
\begin{proof}
Let us denote   $E=\{x_1,x_2,\dots,x_N\}$. Since $u$ is $C^1$, there is strictly increasing $\phi:[0,\infty)\to [0,\infty)$, which is continuous and  $\phi(0)=0$, such that 
\begin{equation}\label{eq:growth}
|u(x)-u(x_i)|\leq \phi(|x-x_i|)|x-x_i|\,.
\end{equation}

Let us choose $\eta>0$ and  an auxiliary comparison function 
\begin{equation}\nonumber
f(x):=4\phi(2\eta)\sum_{i=1}^{N}(2-|x-x_i|)\,=:\,\sum_{i=1}^{N}f_i(x)\,.
\end{equation}

We again denote $g(x):=|\ut(x)-u(x)\,|$, and  show that
for $\eps$ small enough it holds that $g\leq f$.   
Assume that the claim is not true, so that for any $\kappa>0$ there is $x_0\in B(0,1)$
\begin{equation}\nonumber\label{eq:disc-counter}
g(x_0)-f(x_0)+\kappa\geq M:=\sup_{x\in B(0,1)}(g(x)-f(x))>0\,.
\end{equation}
 
 In the case $\dist(x_0,E)> \eta $ we have $|\nabla u(x_0)|\geq c>0$,
and the desired contradiction follows for small enough $\eps$ in the same way as in the proof of Theorem \ref{thm:main-1}.

Consider then the case $\dist(x_0,E)\leq \eta$. We may assume that $|x_0-x_1|\leq \eta$ and $|x_0-x_i|\geq 3\eta$ for all $2\leq i\leq N$. 
In this case, let  $\tau*<\tau$ be the first exit time from $B(x_1,2 \eta)$.
Then, using the strong Markov property of a Markov chain, \cite{meynt09} Proposition 3.4.6, and the estimate (\ref{eq:growth}), we have
\begin{align*}
M-\kappa &\leq g(x_0)-f(x_0)=|\ut(x_0)-u(x_0)|-f(x_0)\\
&=|\mathbb{E}^{x_0}[\ut(x_{\tau*})]-u(x_0)|-f(x_0)\\
&\leq \mathbb{E}^{x_0}\big[|\ut(x_{\tau*})-u(x_{\tau*})|\big]+\mathbb{E}^{x_0}\big[|u(x_{\tau*})-u(x_0)|\big]-f(x_0)\\
&\leq \mathbb{E}^{x_0}\big[|\ut(x_{\tau*})-u(x_{\tau*})|\big]+3\phi(2\eta)\eta-f(x_0)\\
&= \mathbb{E}^{x_0}\big[|\ut(x_{\tau*})-u(x_{\tau*})|-f(x_{\tau*})\big]+\mathbb{E}^{x_0}\big[f(x_{\tau*})-f(x_0)\big]+3\phi(2\eta)\eta\,\\
&\leq M +\mathbb{E}^{x_0}\big[f(x_{\tau*})-f(x_0)\big]+3\phi(2\eta)\eta\,.
\end{align*}
Above we estimated $\mathbb{E}^{x_0}\big[|u(x_{\tau*})-u(x_0)|\big]\le |u(x_{\tau*})-u(x_1)|+|u(x_1)-u(x_0)|\le 2\eta \phi(2\eta)+\eta\phi(\eta)\le 3\eta \phi(2\eta)$ by (\ref{eq:growth}).
Since $\kappa$ can be chosen to be arbitrary small, the desired contradiction follows, if we can show that
\begin{equation}\nonumber\label{equ1}
 \mathbb{E}^{x_0}\big[f(x_{\tau*})-f(x_0)\big]+3\phi(2\eta)\eta\,<0\,.
\end{equation}
For this, we compute that
\begin{align*}
&\mathbb{E}^{x_0}\big[f(x_{\tau*})-f(x_0)\big]=\mathbb{E}^{x_0}\big[\sum_{i=1}^Nf_i(x_{\tau*})-f_i(x_0)\big]\\
=&\mathbb{E}^{x_0}\big[f_1(x_{\tau*})-f_1(x_0)\big]+\mathbb{E}^{x_0}\big[\sum_{i=2}^Nf_i(x_{\tau*})-f_i(x_0)\big]\\
=&4\phi(2\eta)\bigg[\mathbb{E}^{x_0}\big[|x_0-x_1|-|x_{\tau*}-x_1|\big]+\mathbb{E}^{x_0}\big[\sum_{i=2}^N |x_0-x_i|-|x_{\tau*}-x_i|\big]\bigg]\\
\leq &4\phi(2\eta)\bigg[-\eta+\sum_{i=2}^N\Big(|x_0-x_i|-\mathbb{E}^{x_0}\big[|x_{\tau*}-x_i|\big]\Big)\bigg]\,,
\end{align*}
implying that the desired inequality follows if 
\begin{equation}\label{equ2}
|x_0-x_i|\le \mathbb{E}^{x_0}\big[|x_{\tau*}-x_i|\big]
\end{equation}
for all $2\leq i\leq N\,$. 

Now, suppose that (\ref{equ2}) is not true, so that for any $\kappa>0$, there is $z_0\in B(x_1,2\eta)$ such that
\begin{equation}\nonumber
0<M_e:=\sup_{z\in B(x_1,2\eta)}\big(|z-x_i|-e(z)\big)\leq |z_0-x_i|-e(z_0)+\kappa\,,
\end{equation}
 where we denoted
  $$
  e(z):=\mathbb{E}^{z}\big[|z_{\tau*}-x_i|\big].
  $$ 
 It holds that
\begin{equation}\nonumber
e(z_0)=\int_{\rn}e(y)d\mu_{z_0}(y).
\end{equation}
We proceed with the similar reasoning as before
\begin{align*}
M_e \leq &|z_0-x_i|-e(z_0)+\kappa\,=\,|z_0-x_i|-\int_{\rn}e(y)d\mu_{z_0}(y)\,+\kappa\\
=& |z_0-x_i|-\int_{\rn}|y-x_i|d\mu_{z_0}(y)\,+\int_{\rn}\big(|y-x_i|-e(y)\big)d\mu_{z_0}(y)\,+\kappa\\
\leq & |z_0-x_i|-\int_{\rn}|y-x_i|d\mu_{z_0}(y)+M_e+\kappa\,.
\end{align*}
Observe above that if $y$ is outside $B(x_1,2 \eta)$, then  $|y-x_i|-e(y)=0<M_e$.
By substracting $M_e$ from the inequality, the desired contradiction follows by Proposition \ref{aux1}, which guarantees that 
\begin{equation}\nonumber
 |z_0-x_i|<\int_{\rn}|y-x_i|d\mu_{z_0}(y)-\kappa\,
\end{equation}
for all $z_0\in B(x_1,2\eta)\,$. 
\end{proof}

%

\noindent {\bf Acknowledgements.} The second author was supported by the Academy of
Finland project \#260791. The authors would like to thank C.\ Geiss for useful discussions.

\def\cprime{$'$} \def\cprime{$'$} \def\cprime{$'$}


\end{document}